\documentclass[12pt]{amsart}

\usepackage{amsmath, amssymb, amsthm, hyperref}
\usepackage{tikz, booktabs, multirow, graphicx, caption, subcaption, float, mathtools}
\usepackage{hyperref}
\hypersetup{
    colorlinks=true,
    linkcolor=blue,
    filecolor=magenta,      
    urlcolor=cyan,
}
\usepackage[margin = .8in]{geometry}

\theoremstyle{plain}
\newtheorem{theorem}{Theorem}[section]
\newtheorem{lemma}[theorem]{Lemma}
\newtheorem{observation}[theorem]{Observation}
\newtheorem{corollary}[theorem]{Corollary}

\newtheorem{prop}[theorem]{Proposition}
\newtheorem{fact}{Fact}
\theoremstyle{definition}

\newtheorem{example}[theorem]{Example}

\newcommand{\F}{\mathbb{F}}
\def\w{\omega}
\def\FF{\mathbb{F}}
\newcommand{\Z}{\mathbb{Z}}

\def\d{\delta}
\def\s{\sigma}

\DeclareMathOperator{\Crk}{Crk}
\DeclareMathOperator{\wCrk}{wCrk}
\DeclareMathOperator{\PGL}{PGL}
\def\unfrac#1#2{#1/#2}
\def\upnfrac#1#2{(#1)/#2}
\def\twocycle#1,#2{(#1 \, #2)}
\def\threecycle#1,#2,#3{(#1 \, #2\, #3)}
\def\fourcycle#1,#2,#3,#4{(#1 \, #2\, #3 \,#4)}
\def\fivecycle#1,#2,#3,#4,#5{(#1 \, #2\, #3 \,#4 \,#5)}

\title{Permutation polynomials: iteration of shift and inversion maps over finite fields}

\author{Anna Chlopecki}
\email{annanc2@illinois.edu}
\address{Department of Mathematics\\
University of Illinois at Urbana-Champaign\\
Champaign, Illinois\\
USA}

\author{Juliano Levier-Gomes}
\email{jl0128@westminstercollege.edu}
\address{Department of Mathematics\\
Westminster College\\
Salt Lake City, UT\\
USA}

\author{Wayne Peng} 
\email{junwen.wayne.peng@gmail.com}
\address{Department of Mathematics\\
University of Rochester\\
Rochester, NY\\
USA}

\author{Alex Shearer}
\email{sheareaj@plu.edu}
\address{Pacific Lutheran University\\
Tacoma, WA\\
USA}

\author{Adam Towsley}
\email{adtsma@rit.edu}
\address{School of Mathematical Sciences\\
Rochester Institute of Technology\\
Rochester, NY\\
USA}


\begin{document}
\begin{abstract}
    We show that all permutations in $S_n$ can be generated by affine
    unicritical polynomials. We use the $\PGL$ group structure to
    compute the cycle structure of permutations with low Carlitz rank.
    The tree structure of the group generated by shift and inversion
    maps is used to study the randomness properties of permutation
    polynomials.
\end{abstract}

\maketitle

\section{Introduction}

Throughout this paper, let $p$ be an odd prime and $1\leq d< p{-}1$
be an integer coprime to $p{-}1$. A polynomial $f: \F_p \rightarrow
\F_p$ is a \emph{permutation polynomial} if $f$ is bijective, i.e.,
every element in the image will only have one preimage. A permutation
polynomial $f$ represents an element $\Sigma$ in the symmetric group
$S_p$ on $p$ letters, but in general $f$ is not unique in representing
$\Sigma$. For any permutation $\Sigma\in S_p$ of a finite field $F_p$,
we can always find a unique permutation polynomial $f$ with $\deg(f)<p$
to represent $\Sigma$. We then say that the 
\emph{correspondent permutation
polynomial} of $\Sigma$ is~$f$, and say that $f$ represents a permutation
$\Sigma$ if we have no restriction on the degree of $f$.

Given a monic polynomial $f_{d,c}(x) = x^d + c$ over a finite field
$\F_p$, the polynomial is a permutation polynomial when $f_{d,c}$
is a bijection of $\FF_p$. The map $f_{d,c}:\FF_p \to \FF_p$ is a
bijection when $\gcd(d,p-\nobreak 1) =\nobreak 1$. 
In this case $\FF_p$ has no $d$-th
root of unity and, thus, $f_{d,c}$ is injective on a finite set. In
this paper, the authors study the structure of the group generated by
the above monic permutation polynomials; it turns out that the group
can be generated by only two polynomials, the shifting map $\s=x+1$
and the inversion map $\d=x^{p-2}$. We call $x^{p-2}$ the inversion
map
because $x^{p-1}\equiv x^{p-2}\cdot x\equiv 1\bmod p$ for all $x\neq 0$
by Fermat's little theorem. Our main result is
the following (see Theorem~\ref{main_theorem}):

\begin{theorem}\label{generator_theorem}
Let $f_{d,c} = x^d + c$ over $\F_p$. Then
\[
\langle \s,\d\rangle\cong \langle f_{d,c} \rangle \cong
\begin{cases}
    S_p & p \equiv 1 \bmod 4,\\
    A_p & p \equiv 3 \bmod 4,
\end{cases}
\]
where $S_p$ is the symmetric group on $p$ letters, and $A_p$ is the
alternating group on $p$ letters.
\end{theorem}

As a consequence (see Corollary~\ref{Carlitz}) of the theorem, we
revisit a well-known result of Carlitz~\cite{carlitz}, where he shows
that the group generated by affine maps $ax+b$ and the
inverse map $x^{p-2}$ for $a,b\in \FF_p$ is $S_p$.

Since, for this paper, the inversion and the shifting maps 
serve as the only two generators, 
we are able to analyze and
systematically compute permutation polynomials with a
tree, which is almost isomorphic to a $(p{-}1)$-ary tree; 
see Figure~\ref{tree}. We are able to find a general form in terms of $\s$
and $\d$ for the permutation polynomial that represents the permutation
$\twocycle0,1\twocycle2,3$ for any prime $p$.
\looseness=-1

This result is slightly stronger than the ones given
in \cite{Aksoy-Esen-Csmelioglu-Meidl-Topuzoglu-2009} and
\cite{cesmelioglu-ayca-meidl-topuzoglu-2014}. They observed that any
permutation over $\F_p$ can be rewritten as the polynomial of the form
\begin{equation}\label{eq:1}
P_n(x) =
(\cdots((a_0x+a_1)^{p-2}+a_2)^{p-2}\cdots+a_n)^{p-2}+a_{n+1},\quad\text{for
some } n \geq 0.
\end{equation}
In regards to our work, any permutation over $\F_p$ can be represented
by the polynomial $P_n(x)$ with $a_0 = 1$ or $\pm 1$ for $p\equiv 1$
or $3\bmod 4$ respectively. We let $Q_n(x)$ be the polynomial in the
form \eqref{eq:1} with $a_0=1$, we refer to $Q_n$ as the weak form
of index $n$, and $P_n$ is a standard form of index $n$.

Note that the \emph{Carlitz rank} of a permutation $\Sigma$
or its correspondent permutation polynomial $f$, denoted
as $\Crk(\Sigma)$ and $\Crk(f)$ respectively, defined in
\cite{Aksoy-Esen-Csmelioglu-Meidl-Topuzoglu-2009}, is the minimal
number $n$ of inversions required in the form \eqref{eq:1} to
represent $\Sigma$. We define the \emph{weak Carlitz rank}
of a permutation $\Sigma$ or its correspondent permutation
polynomial $f$, denoted as $\wCrk(\Sigma)$ and $\wCrk(f)$
respectively, as the minimal $n$ such that there is a permutation
polynomial $Q_n$ representing $\Sigma$. An obvious relation
between these two ranks is 
$$\Crk(\Sigma)\leq \wCrk(\Sigma)\quad \text{for any }
\Sigma\in S_p.$$
The key lemma of this paper implies
$$
\Crk(\twocycle1,2\twocycle3,4)\leq \wCrk(\twocycle1,2\twocycle3,4)\leq 6
\quad \text{for all }p.$$ 
See~\cite{A-N-O-A-P-V-Q-L-S-A-T-A-2018}
and \cite{Topuzoglu-Alev-2014} for a survey about the recent
development on the Carlitz rank. In Section 4, we will provide an
algebraic geometry point of view on low Carlitz rank and use it to
show that certain permutations cannot appear when the Carlitz rank is low.

Permutation polynomials are an active area of research due to their
valuable applications across the fields of cryptography, engineering,
coding theory, and other fields of math. These polynomials are valuable
because they share properties with genuinely random mappings, while
in special cases also having predictable underlying structures. One
notable result is Polard's rho algorithm for factoring large numbers,
which operates on points of collision using psuedorandom 
functions~\cite{pollard}. Due to quadratic permutation polynomials having a
number of periodic points equal to the expected amount for a random
function, they are a clear choice for use in the rho algorithm
\cite{flynn-garton}.
In cryptography, permutation polynomials have been used to generate
balanced binary words \cite{Laigle-Chapuy}; they also play a key role
in the RC6 encryption algorithm \cite{singh_maity}.

Psuedorandom permutations have found applications in affecting
the efficiency of turbo codes. Though random permutations seem
to be standard practice; more research into semirandom and
nonrandom permutations may produce wanted results in the future
\cite{dolinar_divsalar}.

An interesting property of permutation polynomials is that they do not behave in a manner that is truly random. As shown in \cite{goncarov}, the number of cycles of a given length of a permutation polynomial is bounded. Our result also provide an example of a specific compositions of functions that corresponds to the same element of $S_p$, for all  $p$ greater than a given value.
Hence,
a natural question to ask is then what other nonrandom behaviors
can be induced on families of permutation polynomials. See
\hyperlink{appa}{the Appendix} for an example of this.

\begin{example}
    Let $p = 5$, then the valid choices for $d$ are $1$ and $3$.
        \begin{gather*}
        \begin{array}{c|c|c|c|c|c}
            & x & x+1 & x+2 & x+3 & x+4\\\hline
            f(0) & 0 & 1 & 2 & 3 & 4\\
            f(1) & 1 & 2 & 3 & 4 & 0\\
            f(2) & 2 & 3 & 4 & 0 & 1\\
            f(3) & 3 & 4 & 0 & 1 & 2\\
            f(4) & 4 & 0 & 1 & 2 & 3\\
            \sigma & \text{id} & \fivecycle0,1,2,3,4 &
            \fivecycle0,2,4,1,3 & \fivecycle0,3,1,4,2 &
            \fivecycle0,4,3,2,1
        \end{array}
        \\[10pt]
        \begin{array}{c|c|c|c|c|c}
            & x^3 & x^3+1 & x^3+2 & x^3+3 & x^3+4\\\hline
            f(0) & 0 & 1 & 2 & 3 & 4\\
            f(1) & 1 & 2 & 3 & 4 & 0\\
            f(2) & 3 & 4 & 0 & 1 & 2\\
            f(3) & 2 & 3 & 4 & 0 & 1\\
            f(4) & 4 & 0 & 1 & 2 & 3\\
            \sigma & \twocycle2,3 & \fourcycle0,1,2,4 &
            \twocycle0,2\threecycle1,3,4 & \twocycle0,3\threecycle1,4,2
            & \fourcycle0,4,3,1
        \end{array}
        \end{gather*}
\end{example}

\section{Some Permutation Polynomial Results}
\begin{observation}\label{shift}
The function $x+1$ over $\F_{p}$ corresponds to the even permutation
$$(0\; 1\; \cdots\; p-1),$$
where $p$ is a prime greater than $2$.
\end{observation}


\begin{observation}\label{shifting}
    The permutation associated to $f_{d,c}(x)$ over $\F_{p}$ is
    equal to $\fourcycle0,1,{\cdots},{p{-}1}^{c-i}\sigma_i$ for any
    $\sigma_i$, where $\sigma_{i}$ is a permutation corresponding
    to $f_{d,i}(x)$.
\end{observation}

\begin{proof}
Consider the function $f_{d,c} =  x^d+c = (x^d+i)+c-i$. Viewed as
a composition of functions, this can be written as $ f_{1,c-i} \
\circ \ f_{d,i} $. We have already established that  $f_{d,i}$ is
equivalent to a permutation, which we will refer to as $\sigma_i$, and
from Observation \ref{shift}, we know that $f_{1,c-i}$ corresponds to
$c-i$ iterations of $\fourcycle0,1,{\cdots},{p{-}1}$. By translating
$ f_{1,c-i}  \circ  f_{d,i} $ to cycle notation, we arrive at
$\fourcycle0,1,{\cdots},{p-1}^{c-i}\sigma_i$.
\end{proof}

\begin{prop}\label{num_cosets}
    The number of left (or right) cosets of the cyclic group $\langle
    \fourcycle0,1,{\cdots},{p-1}\rangle$ in the group $\langle
    f_{d,c}\rangle$ over $\F_p$ is $\phi(p-1)$, where $\phi$ is
    Euler's totient function.
\end{prop}

\begin{proof}
    Let's consider functions $f_{d,c}(x)$ over a finite field
    $\F_p$. Let $\tau_p$ be the set of cosets yielded. We know that
    the $\gcd(d,p-1)=1$, thus $|\tau_p|$ is at most $\phi(p-1)$. We
    want to show that $|\tau_p| = \phi(p-1)$. This is equivalent
    to showing that $[\sigma_{i}] \neq [\sigma_{j}]$ for $d_{i}
    \not\equiv d_{j} \bmod (p-1)$.

    Assume that $d_i \neq d_j$ and, for a contradiction, assume that
    $[\sigma_{i}] = [\sigma_{j}]$. Then, $\sigma_{i} \in
    [\sigma_{j}]$. So, there exists $c$ such that $x^{d_i} + c \equiv
    x^{d_j} \bmod (p-1)$ for all $x \in \F_{p}$.
    Then, for every $x \in \FF_p$ it follows that $x^{d_i}
    + c - x^{d_j} \equiv 0 \bmod (p-1)$. We can only have
    $\max(d_{i},d_{j})$ roots, unless $d_i = d_j$ and $c = 0$. If our
    assumption is true, $x^{d_i}+c-x^{d_j}$ must have $p$ distinct
    roots in $\F_{p}$. However, $\max(d_{i},d_{j}) < p$.
    Thus we have a contradiction.
\end{proof}

For a fixed $d$, as a consequence of Observation \ref{shifting},
there are $p$ distinct permutations given by a function $f_{d,c}(x)$
over $\F_{p}$. So, using Proposition \ref{num_cosets}, it is clear
that the total number of permutations yielded by $f_{d,c}(x)$ for
all $d$ is $p\cdot\phi(p-1)$.

\section{Generating $A_p$ or $S_p$ from Permutation Polynomials}

Our overall goal for this section is to observe which groups are
generated by the set of permutations corresponding to $f_{d,c}(x)$
over $\mathbb{F}_p$. For this investigation, we will focus on the
functions $x+1 \pmod{p}$ and $x^{p-2} \pmod{p}$, as both maps behave
in a predictable manner across any choice of $p$. By Observation
\ref{shift}, the function $x + 1$ corresponds to the cycle $(0 \,
1 \, \cdots \, p {-} 1)$.
Also, we notice that $x^{p-2}$ can be simplified via Fermat's
little theorem to $x^{-1}$. This shows that $x^{p-2}$ is its own
inverse. Therefore, it is a permutation of order 2 and must be composed
of disjoint two cycles. From these two properties, our intent is to
show that the functions $x + 1$ and $x^{p-2}$ can act as a minimal
generating set for  $\langle f_{d,c} \rangle$.

It is known from a result of Iradmusa and Taleb \cite{irad_taleb}
that a permutation of the form
$(a \ b)(c \ d)$ and a full cycle, that is a cycle of length $p$,
are sufficient to generate all of the group $A_p$.%
\footnote[1]{This
statement is conditionally true, however all these conditions are
satisfied when $p$ is prime.} 
We have already established that
$x+1$ is a full cycle. We will now show that $(0 \ 1)(2 \ 3)$ can be
generated as a composition of $x+1$ and $x^{p-2}$ for all $p \geq 3$. First, we have the following observation:

\begin{observation}\label{only_three}
In $\FF_p$, for $p \geq 3$, the function $x^{p-2}$ only has three
fixed points, $0$, $1$, and~${-}1$.
\end{observation}

In the following lemma, we will establish a weak form for the
$\twocycle0,1\twocycle2,3$.

\begin{lemma}\label{two_cycle_formula}
    Let $\delta$ and $\sigma$ be the permutations corresponding to
    $x^{p-2}$ and $x + 1$, respectively. Then, for $p \geq 5$,
    \[
        \twocycle0,1\twocycle2,3 =
        \sigma^{3}\delta\cdot\sigma^{-1}\delta\cdot\left(\sigma\delta\right)^{3}\cdot\sigma^{-1}\delta.
    \]
    In particular, we have $\Crk((01)(23))= \wCrk((01)(23))= 6$.
\end{lemma}

The formula above may seem like magic: it was found using a tree
describing the possible nontrivial permutations generated by $\delta$
and $\sigma$. We first apply $\delta$ and then
can apply any of $\sigma, \sigma^2, \ldots, \sigma^{p-1}$.
We can then only apply $\delta$, see Figure \ref{tree}, after which
we can apply $\sigma, \sigma^2, \ldots, \sigma^{p-1}$. Exhausting
this tree led us to a form which always gives the desired result.

\begin{figure}
\scalebox{.83}{
\begin{tikzpicture}[level/.style={sibling distance=50mm/#1}]
\node [circle,draw] (z){$\text{id}$}
child {node [circle,draw] {$\s$}
    child {node [circle,draw] {$\d$}
      child {node [circle,draw] (x) {$\s$}
        child {node [circle,draw] {$\d$}
              child {node {\vdots}}
               }
}
      child {node [circle,draw] (y) {$\s^{p-1}$}
        child {node [circle,draw] {$\d$}
              child {node {\vdots}}
              }
         }
    }
  }
  child {node [circle,draw] (a) {$\s^2$}
    child {node [circle,draw] {$\d$}
      child {node [circle,draw] (c) {$\s$}
        child {node [circle,draw] {$\d$}
              child {node {\vdots}}
               }
}
      child {node [circle,draw] (d) {$\s^{p-1}$}
        child {node [circle,draw] {$\d$}
              child {node {\vdots}}
              }
         }
    }
  }
child {node [circle,draw] (b) {$\s^{p-2}$}
    child [grow=down] {node [circle,draw] {$\d$}
      child {node [circle,draw](j)  {$\s$}
        child {node [circle,draw] {$\d$}
              child {node {\vdots}}
        }
        }
      child {node [circle,draw](k)  {$\s^{p-1}$}
        child {node [grow=down][circle,draw] {$\d$}
              child {node {\vdots}}
              }
        }
      }
}
child {node [circle,draw] {$\s^{p-1}$}
    child [grow=down] {node [circle,draw] {$\d$}
      child {node [circle,draw](e)  {$\s$}
        child {node [circle,draw] {$\d$}
              child {node {\vdots}}
        }
        }
      child {node [circle,draw](f)  {$\s^{p-1}$}
        child [grow=right] {node (t2) {Depth 2} edge from
        parent[draw=none]}
        child {node [grow=down][circle,draw] {$\d$}
              child {node {\vdots}}
              }
        child {node [grow=down] {}edge from parent[draw=none]}
        }
      }
        child [grow=right] {node (t1) {Depth 1} edge from
        parent[draw=none]}
}

;
\path (a) -- (b) node [midway] {$\cdots$};
\path (c) -- (d) node [midway] {$\cdots$};
\path (e) -- (f) node [midway] {$\cdots$};
\path (x) -- (y) node [midway] {$\cdots$};
\path (j) -- (k) node [midway] {$\cdots$};
\end{tikzpicture}
}
\caption{Tree diagram for the group generated by $\d$ and $\s$.}
\label{tree}
\end{figure}
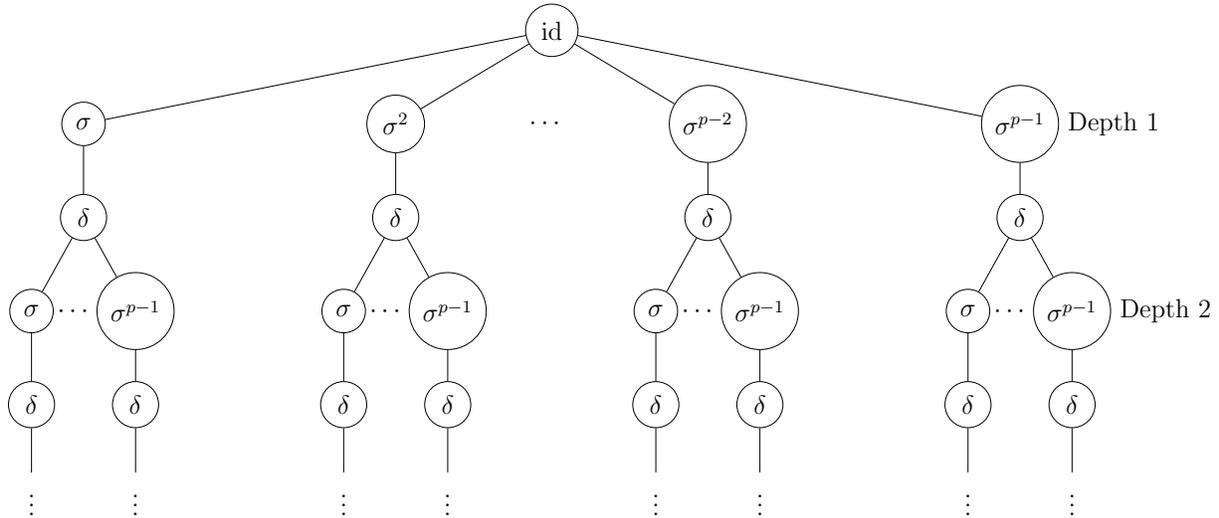

\begin{proof}
    Use $\delta$ and $\sigma$ as defined above, and let
    $c_1,c_2,\ldots<p$ be nonnegative integers. In the following
    argument, we use the following significant but simple fact:
    \begin{fact}\label{fct}
        For any integer $n$ coprime to $p$, we can find an integer
        $0\leq c<p$ such that $\upnfrac{1+cp}{n}$ is not only an integer,
        but also an inverse of $n$, by simply checking
        \[
        n\cdot \dfrac{1+cp}{n}\equiv 1\bmod p.
        \]
    \end{fact}
    We will demonstrate the process of the calculations by showing that
    \[
    \sigma^{3}\delta\cdot\sigma^{-1}\delta\cdot\left(\sigma\delta\right)^{3}\cdot\sigma^{-1}\delta(n)=n
    \]
    for $3<n<p$.

    By Fact \ref{fct}, we have $c_1$ such that
    \[
    \delta(n)\equiv\dfrac{1+c_1 p}{n}\mod p
   \quad \text{and}\quad
    \sigma^{-1}\delta(n)\equiv\dfrac{1+c_1p}{n}-1\equiv\dfrac{1-n+c_1p}{n}.
    \]
    Then, we apply $\delta$ to $\sigma^{-1}\delta$ which gives
    \[
    \delta\sigma^{-1}\delta(n)\equiv \dfrac{1+c_2p}{(1-n+c_1p)/n}\equiv
    \dfrac{n+c_2np}{1-n+c_1p}\mod p
    \vspace*{-5pt}
    \]
    and
    \[
    \sigma\delta\sigma^{-1}\delta(n)\equiv
    \dfrac{n+c_2np}{1-n+c_1p}+1\equiv\dfrac{1+c_3p}{1-n+c_1p}\mod p
    \]
    where $c_3$ is chosen such that $c_3\equiv c_1+c_2n\mod p$. We
    should note that all of these fractions are actually integers. This
    is important because it allows us to use Fact \ref{fct}. Since
    the process repeats the above computations, we will only show a
    few more steps and leave it to the reader to check the rest:
    \begin{align*}
        \delta\sigma\delta\sigma^{-1}\delta(n)&\equiv\dfrac{1+c_4p}{(1+c_3p)/(1-n+c_1p)}\\
        &\equiv \dfrac{1-n+(1-n)c_3p+c_1p+c_1c_4p^2}{1+c_3p}\\
        &\equiv\dfrac{1-n+c_5p}{1+c_3p}\bmod p\quad \text{ for }c_5\equiv
        (1-n)c_3+c_1+c_1c)4p \bmod p;\\
        (\sigma\delta)^2\sigma^{-1}\delta(n)&\equiv\dfrac{1-n+c_5p}{1+c_3p}+1\\
        &\equiv\dfrac{2-n+c_6p}{1+c_3p}\bmod p\quad \text{ for }c_6\equiv
        c_3+c_5\bmod p\\
        \delta(\sigma\delta)^2\sigma^{-1}\delta(n)&\equiv\dfrac{1+c_7}{(2-n+c_6p)/(1+c_3p)}\\
        &\equiv \dfrac{1+c_8p}{2-n+c_6p}\bmod p\quad \text{ for }c_8\equiv
        c_6+(2-n)c_3+c_3c_6p\bmod p\\
        (\sigma\delta)^3\sigma^{-1}\delta(n)&\equiv\dfrac{1+c_8p}{2-n+c_6p}+1\\
        &\equiv \dfrac{1+c_8p+2-n+c_6p}{2-n+c_6p}\\
        &\equiv\dfrac{3-n+c_9p}{2-n+c_6p}\bmod p\quad \text{ for }c_9\equiv
        c_6+c_8\bmod p\\
        &\ \; \vdots
    \end{align*}

    In summary, the following chart proves the lemma. For $3 < n < p$,
    \[
        \begin{array}{ccccccccc}
            0 & \xrightarrow{\delta} & 0 & \xrightarrow{\sigma^{-1}}
            & -1 & \xrightarrow{\delta} & \cdots &
            \xrightarrow{\sigma^{3}} & 1 \\
            1 & \xrightarrow{\delta} & 1 & \xrightarrow{\sigma^{-1}}
            & 0 & \xrightarrow{\delta} & \cdots &
            \xrightarrow{\sigma^{3}} & 0 \\
            2 & \xrightarrow{\delta} & \frac{pc_{2,1} + 1}{2} &
            \xrightarrow{\sigma^{-1}} & \frac{pc_{2,2} - 1}{2} &
            \xrightarrow{\delta} & \cdots & \xrightarrow{\sigma^{3}}
            & 3 \\
            3 & \xrightarrow{\delta} & \frac{pc_{3,1} + 1}{3} &
            \xrightarrow{\sigma^{-1}} & \frac{pc_{3,2} - 2}{3} &
            \xrightarrow{\delta} & \cdots & \xrightarrow{\sigma^{3}}
            & 2 \\
            n & \xrightarrow{\delta} & \frac{pc_{n,1} + 1}{n} &
            \xrightarrow{\sigma^{-1}} & \frac{pc_{n,2} + 1 - n}{n} &
            \xrightarrow{\delta} & \cdots & \xrightarrow{\sigma^{3}}
            & n
        \end{array}
    \]
    where $1 \leq c_{i,j} < p$ are integers.

    The $\Crk(\Sigma)=6$ follows directly from Lemma 2 in
    \cite{A-N-O-A-P-V-Q-L-S-A-T-A-2018}, which says that if $P_n$
    represents $\Sigma$ with $n<(p-1)/2$, then $n=\Crk(\Sigma)$.
\end{proof}

One more observation and lemma, we can show our main results.
\begin{observation}\label{odd}
We have the following:
    \begin{enumerate}
        \item The polynomial $x^{p-2}$ gives an odd permutation of\/
        $\FF_p$ if and only if $p \equiv 1 \bmod 4$.
        \item The polynomial $-x$ gives an odd permutation of\/ $\FF_p$
        if and only if $p \equiv 3 \bmod 4$.
    \end{enumerate}
\end{observation}

\begin{proof}

    (1)\, Under the inverse map $x^{p-2}$, by Observation
    \ref{only_three}, we know that there are only $3$ fixed points:
    $0$, $1$, and $-1$. Moreover, the permutation correspondent to
    $x^{p-2}$ can be written as disjoint $2$-cycles. Therefore, the
    map $x^{p-2}$ must produce a permutation with $\upnfrac{p-3}{2}$
    transpositions.

    If $p \equiv 1 \bmod 4$ then $p = 4m +1$ for some $m \in
    \Z$. Therefore, there are $2m -1$ transpositions. So, there are
    an odd number of transpositions; therefore $x^{p-2}$ forms an
    odd permutation. Hence, there exists an odd permutation in the
    set given by $f_{d,c}(x)$ over $\F_p$ if $p \equiv 1 \bmod 4$.
\smallbreak

\noindent (2)\, The map $-x$ will have one fixed point, 0. Everything
else will be sent to its additive inverse. This means that $-x$
yields a permutation made up of transpositions. There must be
${\upnfrac{4m+3-1}{2}}$ transpositions within our permutation, or $2m-1$
transpositions. So $-x$ yields an odd permutation.
\end{proof}

\begin{lemma}\label{even}
    There exist no odd permutations in the set given by $f_{d,c}(x)$
    over $\F_p$ if $p \equiv 3 \bmod 4$.
\end{lemma}

\begin{proof}
    In order to show that no odd permutations are generated, we need
    only consider $f_{d,0}$, as a shift is an even permutation by
    Observation \ref{shift}. We will show that each permutation is
    a product of an even number of cycles. We claim that for each
    cycle $\threecycle{a_1},{\cdots},{a_k}$ in a given permutation,
    there exists the cycle $\threecycle{-a_1},{\cdots},{-a_k}$ in
    the same permutation.

    We first want to show that if $a_i \neq 0$ is in cycle
    $\threecycle{a_1},{\cdots},{a_k}$, then $-a_i$ is not in this
    cycle. In other words, for all $d$, $a_i^d \neq -a_i$. Let's
    consider a permutation $\sigma$. Suppose, for the sake of
    contradiction, $a_i^{d^n} = -a_i$ for arbitrary $a_i \in
    \F_p$. Then, $a_i
    ( a_i^{d^{n}-1} + 1) = 0$. Since we are considering nonzero $a_i$,
    we have $a_i^{d^{n}-1} + 1 = 0$. Hence, $(a_i^{(d^{n}-1)/2})^2
    \equiv -1 \bmod p$. However, because $ p \equiv 3 \bmod 4$,
    by quadratic reciprocity
    (see, for example, \cite[Chapter~21]{Silverman_friendly}),
    there is no solution to this equation.

    Because $d$ is odd, it follows that if $a_i^d = a_{i+1}$, then
    $-a_{i}^d = -a_{i+1}$. Thus, for each cycle
$\threecycle{a_1},{\cdots},{a_k}$, we have a matching cycle
$\threecycle{-a_1},{\cdots},{-a_k}$ in the same permutation. Thus,
every permutation generated is even.
\end{proof}

\begin{theorem}[Theorem~\ref{generator_theorem}]
\label{main_theorem}
Let $f_{d,c} = x^d + c$ over $\F_p$. Then
\[
\langle \s,\d\rangle\cong \langle f_{d,c} \rangle \cong
\begin{cases}
    S_p & p \equiv 1 \bmod 4\\
    A_p & p \equiv 3 \bmod 4
\end{cases},
\]
where $S_p$ is the symmetric group on $p$ letters, and $A_p$ is the
alternating group on $p$ letters.
\end{theorem}

\begin{proof}
    For $p \geq 5$ it follows from Lemma \ref{two_cycle_formula},
    and for $p= 3$ it follows from a direct computation in SageMath \cite{sagemath},
    that the set of permutations generated by $f_{d,c}(x)$
    over $\F_{p}$ must contain $\twocycle0,1\twocycle2,3$ and
    $\fourcycle0,1,{\cdots},{p{-}1}$.
    Thus, by \cite{irad_taleb},
    the permutations yielded by $f_{d,c}(x)$ over $\F_{p}$ must
    generate at least $A_p$, as the cycles $\twocycle0,1\twocycle2,3$
    and $\fourcycle0,1,{\cdots},{p{-}1}$ generate $A_p$.
    We know by Lagrange's theorem that $A_p$ and an odd permutation
    generate $S_p$. By Lemma \ref{odd}(1), we know that we can find
    an odd permutation in our generating set when $p \equiv 1 \bmod
    4$. Hence, when $p \equiv 1 \bmod 4$, we can generate~$S_p$.
    By Lemma~\ref{even}, we know that no odd permutations exist within
    our generating set when $p \equiv 3 \bmod 4$. So, when $p \equiv
    3 \bmod 4$, we generate $A_p$.
\end{proof}

The following corollary removes the monic condition on our generating
polynomials $f_{d,c}$. It recovers a result of Carlitz \cite{carlitz}
with an elementary proof.

\begin{corollary}\label{Carlitz}
    Let $f_{a,d,c} = ax^d + c$ over $\F_p$ where $a=1$ if $p\equiv
    1\bmod{4}$ and $a=\pm 1$ if $p\equiv 3\bmod{4}$. Then $\langle
    f_{a,d,c} \rangle \cong S_p$ for all p prime.
\end{corollary}

\begin{proof}
By Theorem \ref{main_theorem}, we know this is true when $p \equiv 1
\bmod 4$. So, we must only consider the case when $p \equiv 3 \bmod 4$.
In this case, by Lemma \ref{odd}(2), we know that $-x$ gives an odd
permutation. So all of $S_p$ is generated.
\end{proof}

\section{Carlitz rank and weak Carlitz rank}
From Corollary~\ref{Carlitz}, we know that any permutation
polynomial can be represented by the form $P_n$ with $a_0=\pm 1$;
see \eqref{eq:1}. We will say $P_n(x)$ is a weak form if $a_0=\pm 1$,
and denote it as $Q_n$.
One can define a rank of a permutation $\Sigma$ or a permutation
polynomial over $\F_p$ with degree $<p$ similar to the Carlitz rank,
called the weak Carlitz rank, as the minimal number $n$ such that $Q_n$
represents~$\Sigma$. We should keep in mind that the definition of this
weak form is not redundant since we naturally have
\[
0=\Crk(ax+b)<\wCrk(ax+b)\quad\text{for}\quad a\neq \pm 1.
\]
Moreover, if we find $0\neq \wCrk(ax+b)\leq (p-1)/2$,
then we have $\Crk(ax+b)\neq 0$ from a result in
\cite{A-N-O-A-P-V-Q-L-S-A-T-A-2018}. Hence, we conclude that
$\wCrk(ax+b)>(p-1)/2$.

In cryptography, there are several different measures for the
complexity of a permutation polynomial. Let us introduce some of these:

We follow the definition given in
\cite{Icsik-Leyla-Winterhof-Arne-2018}. The \emph{linearity}
$\mathcal{L}(f)$ (or $\mathcal{L}(\Sigma)$) of a permutation polynomial
$f$ (or a permutation $\Sigma$) over a finite field $\FF_p$ with
$f(0)=0$ is
\[
\mathcal{L}(f)\coloneqq\max_{a\in \FF_p^{*}}\lvert\{c\mid f(c)=ac\}\rvert.
\]
We say an element $c\in\FF_p$ is $a$-\emph{linear} if $f(c)=ac$. From
this point of view, any $x\in\FF_p$ is $a$-linear for some $a$,
so $\mathcal{L}(f)\geq 1$.

Another canonical measurement is called weight $\w(f)$ of a permutation
polynomial $f$ with $\deg(f)<p$ which is the number of nonzero
coefficient of $f$.

It is worth mentioning that there is another measurement, namely
the index of a permutation polynomial. It was first introduced in
\cite{Niederreiter-Winterhof-2015}, defined in \cite{Wan-Wang-2016},
and further studied in \cite{Mullen-Wan-Wang-2014,Wan-Wang-2016,Wang-2013}. 
For applications in
cryptography, one would like to have a permutation polynomial $f$
that has small linearity and large Carlitz rank, degree, and weight.

Using Lemma~\ref{two_cycle_formula} we can measure the complexity of
the form $\twocycle1,2 \twocycle3,4$.

\begin{prop}
For any permutation $\Sigma$ of the form $\twocycle
a,{a+1}\twocycle{a+2},{a+3}$, we have
\[
\Crk(\Sigma)=\wCrk(\Sigma)= 6,
\quad
\deg(\Sigma)\geq p-7,
\quad
\text{and}
\quad
\w(\Sigma)> \dfrac{p-14}{7}.
\]

\end{prop}
\begin{proof}
The proof of this proposition is just a direct consequence
of results in \cite{Icsik-Leyla-Winterhof-Arne-2018},
\cite{csmelioglu-meidl-topuzoglu-2008} and
\cite{Gomez-Domingo-Ostafe-Topuzoglu-2014}. More precisely, we have
the following inequalities:

\begin{align}
    & \Crk(\Sigma)\geq p-\mathcal{L}(\Sigma);\\
    & \Crk(\Sigma)\geq p-\deg(f)-1;\\
    & \Crk(\Sigma)\geq\dfrac{p}{\w(\Sigma)-2}+1.
\end{align}

Since $\twocycle{a},{a+1}\twocycle{a+2},{a+3}$ is conjugate to
$\twocycle0,1\twocycle2,3$ by an iteration of the shifting
map $\sigma$, the weak Carlitz rank is the same. Thus, we get our
inequality by direct computation.
\end{proof}

It has been observed that $\s^k\d$ is equal to
\[
R(x)=k+\dfrac{1}{x}=\dfrac{k x+1}{x}
\]
for all $x\in \F_p\setminus \{0\}$ where $0$ is the pole of the
mobius transformation. Thus, we can let $P_0(x)=R_0(x)=a_1x+a_0$,
and define the following recursive relation
\[
R_m(x)=a_{m+1}+\dfrac{1}{R_{m-1}(x)}
\]
to have $P_n(x)=R_n(x)$ for $x\in \FF_p\setminus\{\rho_m \mid \rho_m
\text{ is the pole of }R_m\ \forall m=1,2,\ldots n\}$. We denote
$\mathcal{O}^n_{(a_0,\ldots,a_{n+1})}=\{\rho_m \mid\rho_m \text{ is
the pole of }R_m\ \forall m=1,2,\ldots,n\}$, and omit the subindex
$(a_0,\ldots,a_{n+1})$ if it is clear in the context. This alternating
expression explains why one needs a permutation polynomial with large
Carlitz rank for application of cryptography.

\begin{prop}
Given a permutation polynomial $f$ over $\FF_p$, let
$n=\Crk(f)$. Then, there are at least $p-n$ elements $c\in\F_p$
satisfying $f(c)=R_n(c)$ where $R_n$ is the correspondent mobius
transformation. If $n=\Crk(f)<(1-(1/2)^{1/3})p\approx 0.21p$, then the
probability of solving the correspondent $R_n$ by randomly choosing
$3$ points, is greater than $1/2$.
\end{prop}
\begin{proof}
First of all, $|\mathcal{O}^n|\leq n$ since there are $n$
inversions, so there are at least $p-n$ many $c$ satisfying
$f(c)=R_n(c)$. Thus, $p-n\geq p-(1-(1/2)^{1/3})p=(1/2)^{1/3}p$. Hence,
we know the ratio of $c\in\FF_p$ satisfying $f(c)=R_n(c)$ is
$(1/2)^{1/3}$. Hence, the conclusion is directly followed.
\end{proof}
We furthermore have the following canonical isomorphism
\[
\biggl\{\frac{ax+b}{cx+d}\Bigm| a,b,c,d\in \FF_p\biggr\}\cong
\PGL_2(\FF_p)
\]
where the operations are functional composition and matrix
multiplication respectively.

From these two observations, we have the following correspondence
between the form $P_n(x)$ over a finite field $\FF_p$ and matrix
in $\PGL_2(\FF_p)$:
\[
P_n(x)\longleftrightarrow
\begin{bmatrix}
a_{n+1} & 1\\
1& 0\\
\end{bmatrix}
\begin{bmatrix}
a_{n}& 1\\
1 & 0
\end{bmatrix}\cdots
\begin{bmatrix}
a_2& 1\\
1&0
\end{bmatrix}
\begin{bmatrix}
a_1& a_0\\
0 & 1
\end{bmatrix}\bmod{p}.
\]
To find all $a$-linear elements of $P_n(x)$ over $\FF_p$, we need
to solve $P_n(x)=ax$ over $\F_p$. Correspondingly, we will solve an
equation of matrices,
\[
\begin{bmatrix}
a_{n+1} & 1\\
1& 0\\
\end{bmatrix}
\begin{bmatrix}
a_{n}& 1\\
1 & 0
\end{bmatrix}\cdots
\begin{bmatrix}
a_2& 1\\
1&0
\end{bmatrix}
\begin{bmatrix}
a_1& a_0\\
0 & 1
\end{bmatrix}\bmod{p}=
\begin{bmatrix}
a & 0\\
0 & 1
\end{bmatrix}
\]
where we treat all $a_i$ as variables. The solution of the equation
of matrix is a variety $V\in \F_p^{n+2}$ with dimension $\dim(V)\geq
n-2$. Therefore, $V$ is not empty for $n\geq 3$. If $(a_{0},a_1,\ldots,
a_{n+1})$ is on $V$, then the correspondent permutation polynomial
$P_n(x)$ is $ax$ for $x\in \F_p\setminus\mathcal{O}^n$. This is
definitely not a good way to approach this problem for large $n$,
but we can still say something for small $n$ with fixed $a_1$ and
$a_0$. We should remind readers that $a_1=1$ is correspondent to the
weak form~$Q_n$.
\begin{theorem}
For $p \geq 13$, the only permutation polynomial $P_n$ with $n \leq 4$
and $a_1 = \alpha$ that has $p-4$ many $\alpha$-linear points
is the polynomial $\alpha x$.
In particular, we conclude $\wCrk(\Sigma)>4$ if $\Sigma$ permutes at
most $4$ elements.
\end{theorem}
\begin{proof}
If there are $4$ elements not in $\mathcal{O}^4$ on $y=R_4(x)$,
then $R_4(x)$ is completely determined by these points.
Therefore, if we have more than $4+|\mathcal{O}^4|$ $\alpha$-linear
elements, then $R_4(x)=ax$.
Since we assumed $f$ has at least $p-4$ $\alpha$-linear elements,
we have the inequality
\[
p-4\geq4+4\geq 4+|\mathcal{O}^4|,
\]
and we get that $p$ is a prime at greater than $12$.

Thus, our condition implies the following matrix equation
\[
\begin{bmatrix}
a_{5} & 1\\
1& 0\\
\end{bmatrix}
\begin{bmatrix}
a_{4}& 1\\
1 & 0
\end{bmatrix}
\begin{bmatrix}
a_3& 1\\
1&0
\end{bmatrix}
\begin{bmatrix}
a_2& 1\\
1 & 0
\end{bmatrix}
\begin{bmatrix}\alpha & a_0\\
0 & 1
\end{bmatrix}=
\begin{bmatrix}
\alpha & 0\\
0 & 1
\end{bmatrix}
\]
which yields the following system of equations:
\begin{align*}
\alpha a_2a_3a_4a_5+\alpha a_2a_3+\alpha \alpha a_2a_5
+\alpha a_4a_5 & =0\\
a_0a_2a_3a_4a_5+a_0a_2a_3+a_0a_2a_5+a_0a_4a_5+a_3a_4a_5+a_0+a_3+a_5 
& =0\\
a_2a_3a_4+a_4+a_2 & =0\\
a_0a_2a_3a_4+a_3a_4+a_0a_4+a_0a_2 & = 0 
\end{align*}

The first equation and the third equation together imply $\alpha
a_2a_3=0$, which means $a_2=0$ or $a_3=0$. Therefore, the form
$P_4$ with $a_1=\alpha$ will be reduced to the form $P_2$ with
$a_1=\alpha$. We then set up a matrix equation regarding $P_2$
with $a_1=\alpha$, and find it will reduce again to $P_0$ with
$a_1=\alpha$. Hence, the only form $P_n$ with $a_1=\alpha$ and
$n\leq 4$ which has at least $p-4$ many $\alpha$-linear elements is
$\alpha x$.
\end{proof}

This framework can also give us insight with regards to the iteration
of the permutation polynomial $f_{p-2,a}(x)=x^{p-2}+a$. The matrix
correspondent to $f_{p-2,a}$ is
\[
M=
\begin{bmatrix}
a & 1\\
1 & 0
\end{bmatrix},
\]
so the matrix correspondent to the $n$-th iterate of $f_{p-2,a}$,
denote by $f_{p-2,a}^n$, is simply $M^n$. It should be noticed that
\[
M^n=\begin{bmatrix}
F_{n+1}(a)& F_n(a)\\
F_n(a) & F_{n-1}(a)
\end{bmatrix}
\]
where $F_n(a)$ is the $n$-th Fibonacci polynomial. One of the many
identities of Fibonacci polynomials that are going to help us here
is the following
\begin{equation}\label{eq:fibonaccipolynomial}
F_{n+m}=F_{n+1}F_{m}+F_{n}F_{m-1}.
\end{equation}
The other fact which will be used is the closed form of the sequence
$\{F_n(\alpha)\}$ for $\alpha\in\mathbb{N}$ is
\[
F_n=Az_+^{n}+Bz_-^{n}
\]
where we let
\[
z_+=\dfrac{-\alpha+\sqrt{\alpha^2+4}}{2}\quad\text{and}\quad
z_-=\dfrac{-\alpha-\sqrt{\alpha^2+4}}{2},
\]
and let $A=1/(z_+-z_-)$ and $B=-1/(z_+-z_-)$. We say the sequence is
not ramified at a prime $p$ if $\alpha^2+4\neq 0$. If we let $F_n\equiv
0\mod p$, it is equivalent to say
\begin{equation}\label{eq:Fn=0}
\left(\dfrac{z_+}{z_-}\right)^{n}\equiv -\dfrac{B}{A}\equiv 1\pmod{p}.
\end{equation}
The $n$ is the multiplicative order of the element $(z_+/z_1)$ in the finite field $\F_p(\sqrt{\alpha^2+4})$, and so $n$ divides $p^2-1$. One can find more details about the order in \cite{Peng2015}.

\begin{lemma}\label{lemma:fibo}
Let $p$ be an unramified prime for the sequence
$\{F_n(\alpha)\}$, and assume $n,m \in \mathbb{N}$ where $n>m$. If
$F_n(\alpha)F_{m-1}(\alpha)\equiv F_{n-1}(\alpha)F_m(\alpha)\pmod{p}$,
and $n$ is the first integer after $m$ satisfying the equation, then
$n-m$ is the multiplication order of $(z_+/z_-)$ in the ring $\FF_p$
or $\FF_{p^2}$.
\end{lemma}
\begin{proof}
Using the closed form of $F_n(\alpha)$, we have
\[
(Az_+^{n}+Bz_-^{n})(Az_+^{m-1}+Bz_-^{m-1})\equiv(Az_+^{n-1}+Bz_-^{n-1})(Az_+^{m}+Bz_-^{m})\pmod{p}.
\]
The equation then will simplify to
\[
\left(\dfrac{z_+}{z_-}\right)^{n-m}\equiv 1\pmod{p}.
\]
Since we assumed $n$ is the first integer after $m$ satisfying the
equation, $n-m$ should be the multiplication order of $(z_+/z_-)$.
\end{proof}
\begin{theorem}
If $n<p$ is the minimal integer such that  $F_n(\alpha)\equiv 0\mod
p$, then the permutation polynomial $f_{p-2,\alpha}^n$ represents
\[
\biggl(-\frac{F_{n-1}}{F_{n-2}}\; \ldots\; -\frac{F_3}{F_2}\;
-\frac{F_2}{F_1}\; 0\biggr),
\]
Where $F_i$ is evaluated at $\alpha$.
Moreover, we have $n$ dividing $p^2-1$. Conversely, if
$f_{p-2,\alpha}^n$ fixes at least $n+4$ elements in $\FF_p$,
and $2n+4\leq p$, then $F_n(\alpha)\equiv 0\mod p$.
\end{theorem}
\begin{proof}
We have to show $|\mathcal{O}_{(\alpha,\ldots,\alpha)}^n|=n$. It is
equivalent to show that
\[
\dfrac{F_{k+1}}{F_{k}}\neq \dfrac{F_{l+1}}{F_l}
\]
for all $0\leq l\neq  k\leq n$.

By \eqref{eq:fibonaccipolynomial}, we have
\[
F_{n+1}=F_{n}F_2+F_{n-1}F_{1}
\]
Since we assume $F_n(\alpha)=0$,
$F_{n+1}(\alpha)=F_{n-1}(\alpha)$. Therefore, $M^n$ is
equivalent to the identity matrix, which is correspondent to
the identity map. Thus, $R_n(x)=x=f_{p-2,\alpha}^n(x)$ for $x\in
\FF_p\setminus\mathcal{O}_{(\alpha,\ldots,\alpha)}^n$.

We have to show $\#\mathcal{O}_{(\alpha,\ldots,\alpha)}^n=n$. It is
equivalent to show that if
\[
\dfrac{F_{l}}{F_{l+1}}=\dfrac{F_{k}}{F_{k+1}}
\]
for some $0\leq l,k\leq n$, then $l=k$. Without lose of generality,
we can assume $k\geq l$. By Lemma~\ref{lemma:fibo}, we know $k-l$
is the multiplicative order of $z_+/z_-$. However, $n-1$ should
be less than or equal to the order by \eqref{eq:Fn=0} and
our assumption on $n$, so we get $n\leq k-l$. By $0\leq l,k\leq n$,
the consequence has to be $l=k$.

Let $R(x)=R_1(x)$. The pole of $R(x)$ is $0$,
and we have $R_i(x)=R^i(x)$ for all $i$. Given
$c\in\mathcal{O}_{\alpha,\ldots,\alpha}^n $,
$f_{p-2,\alpha}(f^{i}_{p-2,\alpha}(c))=R(R_i(c))=R_{i+1}(c)$ if and
only if $R^{i}_{p-2}(\alpha)(c)\neq 0$. $R(x)$ is equal to $x'/x''$
where $x'$ and $x''$ is given by
\[
\begin{bmatrix}
a & 1\\
1 & 0
\end{bmatrix}
\begin{bmatrix}
x\\
1
\end{bmatrix}
=\begin{bmatrix}
x'\\x''
\end{bmatrix}.
\]

A pole $c$ is $-\unfrac{F_{i-1}}{F_{i}}$ or $R^i$, and the above
computation shows
\[
\begin{bmatrix}
\alpha & 1\\
1 & 0
\end{bmatrix}
\begin{bmatrix}
-F_{i-1}(\alpha)\\
F_i(\alpha)
\end{bmatrix}
=\begin{bmatrix}
-\alpha F_{i-1}(\alpha)+F_i(\alpha)\\
-F_{i-1}(\alpha)
\end{bmatrix}
=\begin{bmatrix}
-\alpha F_{i-1}(\alpha)+\alpha F_{i-1}(\alpha)+ F_{i-2}\\
F_{i-1}
\end{bmatrix}
=\begin{bmatrix}
F_{i-2}\\
-F_{i-1}
\end{bmatrix}.
\]
Thus, we have $f_{p-2,\alpha}^{i-1}(-F_{i-1}/F_i)=0$. It implies that
$f^{n}_{p-2,\alpha}$ represents the permutation
\[
\biggl(-\frac{F_{n-1}}{F_{n-2}}\; \ldots\; -\frac{F_3}{F_2}\;
-\frac{F_2}{F_1}\; 0\biggr).
\]

Conversely, if $f_{p-2,\alpha}^n$ fixes at least $n+4$ elements in
$\F_p$, then it implies at least $4$ elements are $1$-linear and
not a pole of $R_i(x)$ for all $i\leq n$. Therefore, $R_n(x)=x$,
which implies $F_n(\alpha)\equiv 0\pmod{p}$.
\end{proof}

\appendix
\section{Randomness in the $n$-th depth}
\label{appendix}
\hypertarget{appa}{}
\begin{center}
\emph{Appednix by Wayne Peng and Ching-Hua Shih}
\end{center}

While preparing the paper, the authors were asked by Shih-Han Hung  
what we could say
about the permutations which appear in the
$n$-th depth of the inverse tree used in the proof of 
Lemma~\ref{two_cycle_formula}, see Figure \ref{tree}.

We further define the following. On each level of the inverse
tree, we say a permutation is of \emph{the first type} if the last
polynomial used in the composition is the inversion $\s$. Otherwise,
we say the permutation is of \emph{the second type}.

Our belief is that the permutations that appear in the $n$-th level
should be random. If the permutations in the $n$-th level of the
tree do appear randomly, then the probability that two elements,
say $a = 1$ and $b = 2$, appear in the same cycle of a 
randomly selected permutation at that level 
is $\frac12$. If we claim the behavior of permutations of the
first type, $p$, with the correspondent permutation $p\d$ of the
second type, is independent, then the probability that $a$ and $b$
appear in the same cycle for $p$ and $p\d$ together is $\frac{1}{4}$.

Using this fact, we used SageMath \cite{sagemath} to find the frequency of
$b$ in the orbit of $a$ under a permutation. We searched through all primes
from $547$ (the 101-st prime) to $1229$ (the 201-st prime) with depth
from $1$ to $10$. However, it is computationally difficult to go
through all branches on the inverse trees due to the exponential
growth of the trees, so for each prime and each depth, we generated
$500$ random paths to the $n$-th level of the tree and tested
whether $2$ was in the orbit of $1$.

The complete data can be found
\href{https://sites.google.com/site/atowsley42/data}{here}.
We demonstrate
our result by providing histograms for $1$ and $2$ occurring in the same cycle
for polynomials of the first type, the second
type, or for both  types; see Figure~\ref{fig:two}. These histogram support our hypothesis that
the random permutations appear on the $n$-th depth of the inverse tree,
and it also support that $p$ and $p\d$ are independent events.

\begin{figure}
\centering
    \includegraphics[width=45mm, height=45mm]{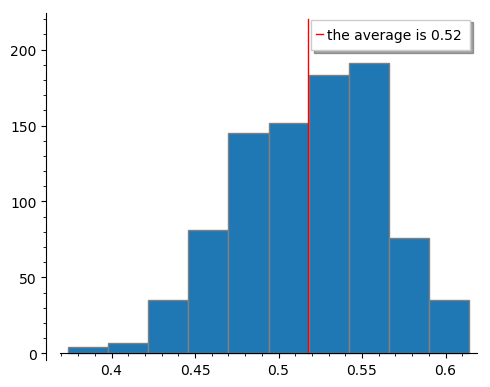}
\hfill
    \includegraphics[width=45mm, height=45mm]{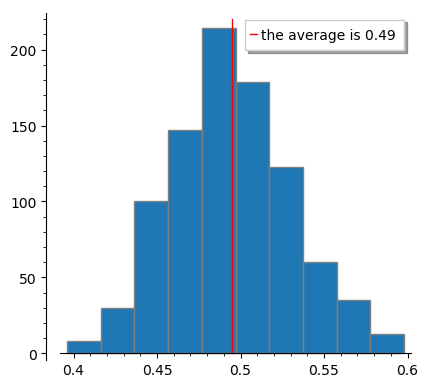}
\hfill
    \includegraphics[width=45mm, height=45mm]{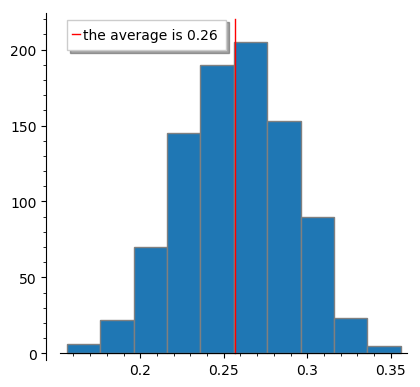}
\caption{Frequency of 2 and 1 in the same cycle for permutations of 
the first type (left), second type (center) and both types (right).}
\label{fig:two}
\end{figure}

We can do a more detailed analysis by applying $p$-test to small
primes and depths. In those cases, we can run through all branches
of the inverse trees by a defined pattern, and get a string of $0$'s
and $1$'s where $0$ means $a=1$ and $b=2$ are not on the same cycle
of the permutation, and $1$ means $a$ and $b$ are on the same
cycle. For a given prime $p$, the branches at depth $d$ can be
indexed by $\{1,2,\cdots,p{-}1\}^d$ with lexicographic order. Let
$(i_1,i_2,\ldots,i_d)$ be an index. The $k$-th coordinate $i_k$
indicates the path, compose with $\sigma^{i_k}$, from the depth
$k-1$ to the depth $k$. The result of the test is shown below,
see Figures~\ref{figone} and~\ref{figtwo}.
The test shows that a string is likely to 
consists only of 0's or only of 1's when the depth becomes large.

\begin{figure}
\begin{tabular}{||c|c||c|c|c||}
\toprule

Prime & Depth & Two sided p-value & Less p-value & Greater p-value\\
\hline
\multirow{4}{*}{3} & 2 & 1 & 0.5 & 0.5\\
                   & 3 & 0.28009 & 0.14004 & 0.85996\\
                   & 4 & 0.78151 & 0.60924 & 0.39076\\
                   & 5 & 0.91628 & 0.54186 & 0.45814\\
\hline
\multirow{4}{*}{5} & 2 & 0.26335 & 0.86833 & 0.13167\\
                   & 3 & 0.45392 & 0.22696 & 0.77304\\
                   & 4 & 0.59871 & 0.70064 & 0.29936\\
                   & 5 & 0.03095 & 0.98452 & 0.01547\\
\hline
\multirow{4}{*}{7} & 2 & 0.16883 & 0.91558 & 0.08442\\
                   & 3 & 0.27459 & 0.86271 & 0.13729\\
                   & 4 & 0.00042 & 0.99979 & 0.00021\\
                   & 5 & 4.42e-05 & 0.99998 & 2.21e-05\\
\hline
\multirow{4}{*}{11} & 2 & 0.68766 & 0.34383 & 0.65617\\
                    & 3 & 0.96205 & 0.48103 & 0.51897\\
                    & 4 & 0.00677 & 0.00339 & 0.99661\\
                    & 5 & 1.13e-05 & 5.65e-06 & 1.00000\\
\hline
\multirow{4}{*}{13} & 2 & 0.73800 & 0.63100 & 0.36700\\
                    & 3 & 0.99996 & 0.50002 & 0.49998\\
                    & 4 & 0.95143 & 0.47572 & 0.52428\\
                    & 5 & 0.56364 & 0.71818 & 0.28182\\
\hline
\multirow{4}{*}{17} & 2 & 0.98750 & 0.50625 & 0.49375\\
                    & 3 & 0.31958 & 0.15979 & 0.84021\\
                    & 4 & 0.66922 & 0.66539 & 0.33461\\
                    & 5 & 0.00416 & 0.00208 & 0.99792\\
\hline
\multirow{4}{*}{19} & 2 & 0.37178 & 0.81411 & 0.18589\\
                    & 3 & 0.60601 & 0.30301 & 0.69699\\
                    & 4 & 0.83950 & 0.58025 & 0.41975\\
                    & 5 & 2.72e-17 & 1.36e-17 & 1\\
\hline
\multirow{4}{*}{23} & 2 & 0.52573 & 0.26286 & 0.73714\\
                    & 3 & 0.93955 & 0.46978 & 0.53022\\
                    & 4 & 0.34180 & 0.82910 & 0.17090\\
                    & 5 & 0.00016 & 7.90e-05 & 0.99992\\
\bottomrule
\end{tabular}

\caption{\label{figone} $p$-values of the for the randonmness of the sequence of zeros and ones created by testing if 2 is in the orbit of 1 modulo $p$ for polynomials of the first type. }
\end{figure}

\begin{figure}
\begin{tabular}{||c|c||c|c|c||}
\toprule

Prime & Depth & Two sided p-value & Less p-value & Greater p-value\\
\hline
\multirow{4}{*}{3} & 2 & 1 & 0.5 & 0.5\\
                   & 3 & 0.28009 & 0.14004 & 0.85996\\
                   & 4 & 0.78151 & 0.60924 & 0.39076\\
                   & 5 & 0.91628 & 0.54186 & 0.45814\\
\hline
\multirow{4}{*}{5} & 2 & 0.05686 & 0.97157 & 0.02843\\
                   & 3 & 0.00060 & 0.99970 & 0.00030\\
                   & 4 & 0.00516 & 0.99742 & 0.00258\\
                   & 5 & 7.67729 & 1.00000 & 3.84e-09\\
\hline
\multirow{4}{*}{7} & 2 & 0.01835 & 0.99083 & 0.00917\\
                   & 3 & 0.00852 & 0.99574 & 0.00426\\
                   & 4 & 7.79e-17 & 1 & 3.90e-17\\
                   & 5 & 1.30e-79 & 1 & 6.50e-80\\
\hline
\multirow{4}{*}{11} & 2 & 0.11040 & 0.94480 & 0.05520\\
                    & 3 & 9.85e-09 & 1.00000 & 4.93e-09\\
                    & 4 & 2.25e-39 & 1 & 1.13e-39\\
                    & 5 & 0 & 1 & 0\\
\hline
\multirow{4}{*}{13} & 2 & 0.44849 & 0.77576 & 0.22424\\
                    & 3 & 2.40e-09 & 1.00000 & 1.20e-09\\
                    & 4 & 1.31e-45 & 1 & 6.54e-46\\
                    & 5 & 0 & 1 & 0\\
\hline
\multirow{4}{*}{17} & 2 & 0.05296 & 0.97351 & 0.0264815215218201\\
                    & 3 & 7.79e-13 & 1.00000 & 3.89725348505465e-13\\
                    & 4 & 1.77e-127 & 1 & 8.85e-128\\
                    & 5 & 0 & 1 & 0\\
\hline
\multirow{4}{*}{19} & 2 & 4.19e-13 & 1.00000 & 2.10e-13\\
                    & 3 & 8.41e-24 & 1 & 4.20e-24\\
                    & 4 & 8.29e-157 & 1 & 4.15e-157\\
                    & 5 & 0 & 1 & 0\\
\hline
\multirow{4}{*}{23} & 2 & 4.34e-09 & 1.00000 & 2.17e-09\\
                    & 3 & 7.59e-28 & 1 & 3.79e-28\\
                    & 4 & 1.71e-256 & 1 & 8.53e-257\\
                    & 5 & 0 & 1 & 0\\
\bottomrule
\end{tabular}
\caption{
\label{figtwo}
$p$-values of the for the randonmness of the sequence of zeros and ones created by testing if 2 is in the orbit of 1 modulo $p$ for polynomials of the second type.}
\end{figure}

\section*{Acknowledgements}
This material is based upon work supported by the National Science
Foundation under Grant No. DMS-1439786 while the author was in
residence at the Institute for Computational and Experimental Research
in Mathematics in Providence, RI, during the Summer@ICERM 2019 program
on Computational Arithmetic Dynamics.

Towsley would like to thank John Doyle, Ben Hutz and Bianca Thompson
for their collaboration organizing Summer@ICERM 2019 where this work was done. 

\bibliographystyle{plain}
\bibliography{sources.bib}
\end{document}